\documentclass[12pt]{article}
\usepackage{geometry}
\usepackage{amsfonts}
\usepackage[utf8]{inputenc}

\usepackage[english]{babel}
\usepackage[
backend=biber,
style=alphabetic,
sorting=nyt
]{biblatex}

\usepackage{amsmath, amsthm, amsfonts,amssymb,graphicx,csquotes,mathtools}
\usepackage{authblk}
\usepackage{float}
\usepackage{bigints}
\usepackage{tikz-cd}
\newtheorem{theorem}{Theorem}[section]
\newtheorem{proposition}[theorem]{Proposition}
\newtheorem{lemma}[theorem]{Lemma}

\newtheorem{corollary}[theorem]{Corollary}

\theoremstyle{definition}

\theoremstyle{remark}

\title{Topology of the space of measure-preserving transformations of the circle.}

\author{Houssam Boukhecham$^3$ and Hamza Ounesli$^{1,2}$}
\date{%
    $^1$\small\textit{Scuola Internazionale Superiorie di Studi Avanzati (SISSA), Trieste, Italy.} \vspace{0.2cm}%
    $^2$\textit{Abdus Salam International Centre for Theoretical Physics (ICTP), Trieste, Italy.}\\(\textit{Email: hounesli@sissa.it/hounesli@ictp.it})\\[2ex]%
    $^3$\textit{Universit\'e de Paris 12, LAMA. Paris, France.}\\(\textit{Email:houssam-eddine.boukhecham@u-pec.fr})}

\begin{document}
\maketitle

\begin{abstract}
This paper is dedicated to prove that the space of circle expanding maps of degree 2 preserving Lebesgue measure is an arc-connected space homeomorphic to an infinite-dimensional Lie group whose fundamental group is $\mathbb{Z}$. The techniques involved in the proof are rather unexpected and lead to a formulation of a general conjecture.
\end{abstract}

\section{Introduction and statement of results.}

\noindent One of the classical problems in topology, dynamics, and geometry is studying properties of the group of \textit{diffeomorphisms} of a closed manifold $M$, preserving a given smooth volume form $\omega.$ Questions in terms of the topology of this space, dynamics-rigidity phenomenons, and algebraic properties can be addressed. There has been extensive work in this direction as in \cite{1,2}. In particular, in \cite{3} J.Moser has shown that these groups are locally arc-connected. In this paper, we generalize Moser's result to a space of non-invertible volume preserving maps in dimension 1. More precisely, we consider our manifold to be the circle, and we study the space of $C^1$ orientation preserving uniformly expanding maps of degree 2, preserving the natural volume form on the circle i.e Lebesgue measure. We denote this space by $\Lambda_{Leb}$. Our results suggest that the facts known for volume preserving \textit{diffeomorphism} groups can be extended to spaces of non-invertible volume preserving maps. The only topological information we know about $\Lambda_{Leb}$ is that it is of first category in the space $C^1(S^1,S^1)$ of all $C^1$ maps of the circle, this was shown in \cite{4}.

\noindent Our result shows that $\Lambda_{Leb}$ is indeed arc-connected, with fundamental group $\pi_{1}(\Lambda_{Leb})=\mathbb{Z}$. Moreover, we show that this space is homeomorphic to a natural infinite dimensional Lie group.

\begin{theorem}
The space $\Lambda_{Leb}$ endowed with the $C^1$-topology is homeomorphic to $ T^2\setminus diag(T^2)\times D_{+}(S^1, 0\ \text{is fixed}),$ in particular, $\Lambda_{Leb}$ is arc-connected, and $\pi_1(\Lambda_{Leb})=\mathbb{Z}.$
\end{theorem}

\textit{Remark: We always denote by $D_{+}(S^1)$ the group of circle diffeomorphisms which preserves the orientation and $D_{+}(I,J)$ for the space of orientation preserving interval diffeomorphisms and $D_{+,exp}(I,J)$ for the expanding ones (i.e $f'>1$).  $T^2$ denotes the torus $S^1\times S^1$ }.

\vspace{0.1cm}
\noindent This theorem, as mentioned before, is an extension of Moser result on local arc-connectedness of the group of volume preserving diffeomorphisms. However, our result extends it only in dimension one. Intuitively the result says that for any two Lebesgue preserving expanding circle maps $f,g$ there exists a deformation  between each other $\gamma(t):[0,1]\to\Lambda_{Leb}$ which preserves Lebesgue along the deformation. The fact that the fundamental group is isomorphic to $\mathbb{Z}$ signifies that any deformation is generated by a fixed deformation in $\Lambda_{Leb}$. On the other hand, we show that the space $\Lambda_{Leb}$ is huge in a sense albeit being meagre in $C^1(S^1,S^1)$, as we have partially proven in \cite{ounesli2023existence}. We conjecture that our result can be extended to arbitrary dimensions.

\vspace{0.2cm}
\noindent \textbf{Conjecture.} Let $(M,g)$ be a closed Riemannian manifold and $\omega$ its volume form. The space $\Lambda^{r}_{\omega}(M)$ of $C^1$ expanding $r$-folds of $M$, preserving the volume form, is locally arc-connected.

\section{Proof of the theorem.}
\subsection{Circle expanding maps}

\noindent Denote by $E^{1}(S^1)$ the space of uniformly expanding maps of the circle, and by $\Lambda_{Leb}$ the sub-space of maps $f$ of degree $2$ and preserving the Lebesgue measure $\lambda$ (i.e $f_{*}\lambda=\lambda$) and the orientation. We endow this space with the $C^1$-topology. 
\noindent The circle is seen as the natural quotient space $[0,1]\slash (0\sim 1)$. Circle maps of degree $2$ which are orientation preserving, up to conjugacy with a rotation, can be regarded as interval maps with two full branches (see figure \ref{A representation of a circle map of degree 2 on the unit interval.}).

\begin{figure}[H]
\centering
\includegraphics[scale=0.5]{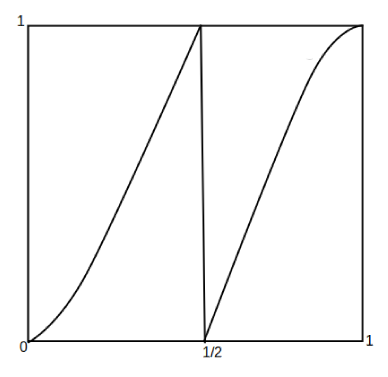}
\caption{A representation of a circle map of degree 2 on the unit interval.}
\label{A representation of a circle map of degree 2 on the unit interval.}
\end{figure}

\noindent We recall that expanding circle maps of degree $2$ have two main characteristics: a unique fixed point $p\in S^1$ and two branch-arcs determined by two distinct points $x_1\neq x_2\in S^1$.

\subsection{The transfer operator.}

\noindent Let $f\in E^1(S^1)$. We define the transfer operator associated to $f$, and acting on $L^1_{\lambda}(S^1)$ as: if $h\in L^1_{\lambda}(S^1)$ then:

\begin{equation}
    Ph=\dfrac{d\big(f_{*}\mu_{h}\big)}{d\lambda}.
\end{equation}

\noindent where $\mu_{h}=h\cdot \lambda$. This operator can be interpreted as the density of the push-forward of measures in respect to Lebesgue. The transfer operator for maps of degree 2 has an explicit formula:

\begin{equation}
    Ph(x)=\sum\limits_{y\in f^{-1}(x)}\dfrac{h(y)}{f'(y)}.
\end{equation}

\noindent The main property of this operator is the following Folklore proposition:

\begin{proposition}
    The set of absolutely continuous invariant measures of $f$ correspond to the fixed points of the operator $P$.
\end{proposition}

 \subsection{Proof of Theorem 1.}

 The proof of the theorem will be based on the following proposition, which on a part we consider to be of independent interest: %maybe lemma

 \begin{proposition}\label{Extension of expanding diffeomorphism to volume preserving expanding map}
Let $a\in(0,1)$ and $f_1:[0,a]\to[0,1]$ be an expanding $C^1$-diffeomorphism, then there exists a unique extension of $f_1$ to a Lebesgue-preserving full branch expanding transformation of the unit interval.%, which lifts to an element of $\Lambda_{Leb}$.
 
 \end{proposition}
\begin{proof}
Consider the differential equation 
\begin{equation}\label{transfer operator equation...}
    \frac{1}{f_1'\big( f_1^{-1}(x)\big)}+\frac{1}{f_2'\big(f_2^{-1}(x)\big)}=1,~x\in[0,1],
\end{equation}
where $f_2:[a,1]\to \mathbb{R}$ is a diffeomorphism into it's image. The equation \eqref{transfer operator equation...} is equivalent to 
\begin{equation}\label{simplified transfer operator equation...}
f_2'(x)=\frac{f_1'\Big(f_1^{-1}\big(f_2(x)\big)\Big)}{f_1'\Big(f_1^{-1}\big(f_2(x)\big)\Big)-1},~x\in[a,1],
\end{equation}
and since $f_1$ is $C^1,$ by Peano's existence theorem the Cauchy problem with the initial condition $f_2(a)=0$ admits a maximal solution $f_2$ defined on the interval $[a,1]$. Let's show that $f_2$ maps diffeomorphically onto $[0,1]$. Notice that $f_2'(x)>1$ for all $x\in [a,1]$, therefore it only remains to show that $f_2(1)=1.$ Assume that $f_2(1)<1,$ and consider $I=[f_{2}(1),1]$. By construction the map $f:[0,1]\to[0,1]$ defined by $f_1$ and $f_2$ preserves Lebesgue measure since $(3)$ corresponds to $(2)$ by taking $h$ to be the constant function 1, so $\lambda(I)=\lambda(f^{-1}(I))=\lambda(f^{-1}_1(I))$ because $f^{-1}_2(I)=\emptyset$, this is a contradiction because $f^{-1}_1$ is a contraction. We conclude that $f$ is indeed a uniformly expanding full branch map of the interval.

\noindent Uniqueness cannot be deduced directly from the equation (\ref{simplified transfer operator equation...}), because Peano's existence theorem provides only existence, we will deduce it using the fact that the solution preserves $\lambda$. Let $f,g:[0,1]\to[0,1]$ be two full branch interval maps which preserve Lebesgue measure, assume they have the same first branches (i,e $f_{1}=g_{1}$) on an interval $[0,a]$, then for every $y\in[0,1]$ we have

\begin{equation*}
    \lambda([0,y])=\lambda(f^{-1}([0,y]))=\lambda(g^{-1}([0,y])),
\end{equation*}

\noindent which implies by assumption that

\begin{equation*}
    \lambda([a,f^{-1}_{2}(y)])=\lambda([a,g^{-1}_2(y)]),
\end{equation*}

\noindent this implies that $f_2^{-1}(y)=g_2^{-1}(y),$ thus uniqueness of solutions.

\end{proof}

\begin{lemma}
    The extension of an expanding diffeomorphism $f_1:[0,a]\to[0,1]$ to a full branch interval map preserving Lebesgue is a $C^1$ circle map, if and only if the following holds:

    \begin{equation}\label{The condition to extend a full branch to an expanding map}
        f_{1}'(0)=\dfrac{f_{1}'(a)}{f_1'(a)-1}
    \end{equation}
\end{lemma}

\begin{proof}
    This is because for a full branch map to lift to a circle map, the derivatives at the end points must coincide, as well as the left and right derivatives at the point $a$, and so by equation $(4)$, we need \eqref{The condition to extend a full branch to an expanding map} to hold.
\end{proof}

\noindent We will use the previous results to show that $\Lambda_{Leb}$ is arc connected.

\begin{corollary}
$\Lambda_{Leb}$ is arc connected.
\end{corollary}
\begin{proof}
Let $f$ be the doubling map of the circle, and $g\in\Lambda_{Leb}.$ Up to composing $g$ with a rotation, we can assume that $g$ and $f$ have the same fixed point $0.$ Denote by $x_g$ the point in $S^1$ such that $\int_0^{x_g}g'(t)~dt=1$, we will construct a homotopy between g and $\tilde{g}$ in $\Lambda_{Leb}$, such that $x_{\tilde{g}}=\frac{1}{2}$. Without loss of generality, let us assume that $x_{g}>\frac{1}{2}$. For $x_{g}>\epsilon>\frac{1}{2}$, translate horizontally the graph of $g|_{(\epsilon,x_{g})}$ to $(\frac{1}{2}-x_{g}+\epsilon,\frac{1}{2})$ by a linear homotopy $T(t,.)$. Now let $z$ close enough to 0, more precisely, chose $z<\frac{1}{2}-x_{g}+\epsilon$. Construct a homotopy $H(t,x)$ as follows: for every $t$ define $H(t,.)|_{[0,z]}=g$ and $H(t,)|_{[\epsilon-t,x_g-t]}=T(t,)$, and for every $t$ extend it in a $C^1$ and expanding way to the whole interval $[0,x_{g}-t]$, as represented on the figure below. This yields a homotopy between $g$ and $\tilde{g}$ in $\Lambda_{Leb}$, because condition \eqref{The condition to extend a full branch to an expanding map} is satisfied for every $t$, also $\tilde{g}$ satisfies $x_{\tilde{g}}=\frac{1}{2}$.

\begin{figure}[H]
\centering
\includegraphics[scale=0.5]{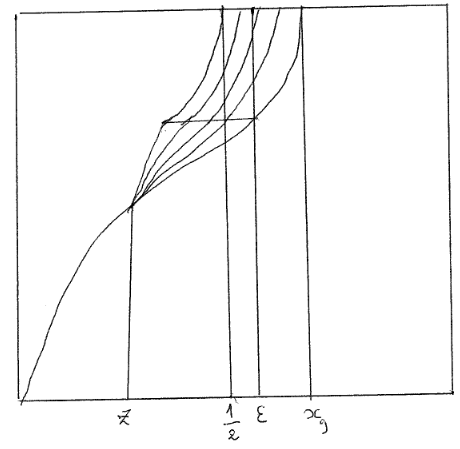}
\caption{A representation of the homotopies $H$ and $T$.}
\end{figure}

\noindent 
\noindent The second step is to construct an appropriate homotopy 
between $\tilde{g}$ and $f$. This is straight forward by considering a continuous family of expanding $C^1$ maps $(h_c:[0,\frac{1}{2}]\to 
[0,1])_{c\in[2,g'(0)]\text{or} [g'(0),2]}$ with $h_c'(0)=c$ and 
$h_c'(\frac{1}{2})=\frac{c}{c-1}.$ Notice in this case that 
$\tilde{g}|_{[0,\frac{1}{2}]}$ is homotopic to $h_{g'(0)}$ by simply 
taking $H(t,x)=t\tilde{g}|_{[0,\frac{1}{2}]}(x)+(1-t)h_{g'(0)}(x)$ and same 
for $f|_{[0,\frac{1}{2}]}$ and $h_{2}$ by $G(t,x)=tf|_{[0,\frac{1}{2}]}(x)+(1-t)h_2(x)$, this homotopies satisfy 
\eqref{The condition to extend a full branch to an expanding map}, and 
so they extend to a homotopy in $\Lambda_{Leb}$ between $\tilde{g}$ and $f$ by concatenating the extension of the homotopy $H$ with the extenstion of the family $(h_c)_c$ and the extension of $G$ in  $\Lambda_{Leb}$, this finishes the proof of arc-connectedness.
\end{proof}

\begin{proposition}
    The space $\Lambda_{Leb}$ is homeomorphic to the infinite dimensional Lie group $T^{2}\setminus diag(T^2)\times Diff(S^1)$.
\end{proposition}
    
\begin{proof}

\noindent Let $\Gamma$ be the space:
\begin{equation*}
    \Gamma=\bigcup\limits_{0\leq x-y<1}\lbrace f\in \text{D}_{+,exp}^1([x,y],[0,1]) \ \text{such that}\ f'(x)=\dfrac{f'(y)}{f'(y)-1}\rbrace.
\end{equation*}
 \noindent Proposition 2.2 results naturally in a map $\mathcal{F}$:
 \begin{equation*}
     \mathcal{F}:\Gamma \to \Lambda_{Leb},
 \end{equation*}
\noindent defined by sending  an element $f\in \Gamma$ to a Lebesgue preserving circle map, by extension after translating $[x,y]$ to $[0,x-y]$, and translating the solution back.

 \begin{proposition}
     The map $\mathcal{F}$ is a homeomorphism (in the $C^1$-topology).
 \end{proposition}

 \begin{proof}
     By proposition 2.2 and lemma 2.3, the map is well defined and for every $f\in \Gamma$, there exists a unique extension of $f$ to a circle expanding map preserving Lebesgue measure. Continuity follows from the fact that the unique solutions to a continuous family of Cauchy problems $(ODE_{t})_{t\in I}$, with a continuous family of initial conditions form a continuous family $(f_{t})_{t\in I}$ in the $C^1$-topology and this shows that $\mathcal{F}$ is a continuous injection.

\noindent The image of the operator $\mathcal{F}$ covers all Lebesgue preserving circle maps $f$, whose fixed point $p_{f}$ is inside the branch interval $[x,y]$ of the specific element, hence it is surjective, the inverse is clearly continuous and hence is a homeomorphism.

\end{proof}

\noindent to finish the proof, notice that $\Gamma$ is homeomorphic to
\begin{equation*}
   (T^2)\setminus diag(T^2)\times \lbrace f\in D_{+}([0,\frac{1}{2}],[0,1])\ \text{such that}\ f(0)=\frac{f(\frac{1}{2})}{f(\frac{1}{2})-1}\rbrace
\end{equation*}

\noindent and that:

\begin{equation*}
    \lbrace f\in D_{+}([0,\frac{1}{2}],[0,1])\ \text{such that}\ f'(0)=\frac{f'(\frac{1}{2})}{f'(\frac{1}{2})-1}\rbrace
\end{equation*}
\begin{equation*}
     \simeq D_{+}([0,1],[0,1]\ \text{such that}\ f'(0)=f'(1))\simeq D_{+}(S^1, 0\ \text{is fixed}).
\end{equation*}

\noindent Now remark that  $T^2\setminus diag(T^2)$ inherits the Lie group structure of $\mathbb{C}\setminus\lbrace0\rbrace$ and $D_{+}(S^1, 0 \ \text{is fixed})$ is an infinite dimensional Lie group. 
 \end{proof}
 \begin{corollary}
 
$\pi_{1}(\Lambda_{Leb})=\mathbb{Z}$.
 \end{corollary}

 \begin{proof}

 First, notice that $\pi_{1}( (T^2)\setminus diag(T^2))=\pi_1(\mathbb{C}\setminus\lbrace 0\rbrace)=\mathbb{Z}$, on the other hand, by results of \cite{5}, we know that the injection of $SO(2)$ in $D_{+}(S^{1})$ induces a splitting of the fundamental group $\pi_{1}(D_{+}(S^{1}))=\pi_{1}(SO(2))\oplus\pi_{1}(D_{+}([0,1],\partial [0,1]))$, and since we know that $\pi_{1}(SO(2))=\mathbb{Z}$, and that $D_{+}([0,1],\partial [0,1])$ is contractible, we deduce that $\pi_{1}(D_{+}(S^{1}))=\mathbb{Z}$ and that $D_{+}(S^1, 0 \ \text{is fixed})$ is simply connected. So we have $\pi_{1}(\Lambda_{Leb})=\mathbb{Z}$.

 \end{proof}

\noindent \textit{Remark. Arc-connectedness can be deduced again by the fact that our space is homeomorphisc to an infinite dimensional Lie group. However, we consider our prove of arc-cnnectedness to be of independent interest since we believe the idea can be generalized to higher dimensions as we conjectured in the statement of results.}

\printbibliography

\end{document}